\documentclass{amsart}
\usepackage{euscript,amsmath, amssymb, amsfonts, mathtools}
\usepackage{mathrsfs}
\usepackage{mathabx}
\usepackage{stmaryrd}

\makeatletter
\@namedef{subjclassname@1991}{Subject}
\makeatother

\numberwithin{equation}{section}
\usepackage{enumitem}   

\newtheorem{theorem}{Theorem}[section]
\newtheorem{lemma}[theorem]{Lemma}
\newtheorem{proposition}[theorem]{Proposition}
\newtheorem{corollary}[theorem]{Corollary}

\theoremstyle{definition}
\newtheorem{definition}[theorem]{Definition}

\newtheorem{algo}{Algorithm}[section]

\theoremstyle{remark}
\newtheorem*{remark}{Remark}
\newtheorem*{remarks}{Remarks}
\usepackage{graphicx}  
\usepackage{icomma}
\usepackage{bbm}
\usepackage{tikz}
\usetikzlibrary{shapes}
\usepackage{color}
\usepackage{algorithmic}
\usepackage{verbatim}
\usepackage{xurl}
\usepackage{indentfirst}
\usepackage{lipsum}
\usepackage[backend = biber]{biblatex}
\usepackage{placeins}
\usepackage{bigints}
\usepackage{subcaption}
\usepackage{appendix}
\DeclareMathOperator{\Span}{Span}
\DeclareMathOperator{\rank}{rank}
\DeclareMathOperator{\CC}{\mathbb C}
\DeclareMathOperator{\RR}{\mathbb R}
\DeclareMathOperator{\ZZ}{\mathbb Z}
\DeclareMathOperator{\NN}{\mathbb N}
\usepackage{csquotes}

\begin{document}
	\title[A Landen-type method for computation of Weierstrass functions]{A Landen-type method for computation of Weierstrass functions}
	
	\author{Matvey Smirnov}
	\address{119991 Russia, Moscow GSP-1, ul. Gubkina 8,
		Institute for Numerical Mathematics,
		Russian Academy of Sciences}
	\email{matsmir98@gmail.com}

        \author{Kirill Malkov}
        \address{119991 Russia, Moscow GSP-1, ul. Gubkina 8,
		Institute for Numerical Mathematics,
		Russian Academy of Sciences}
	\email{markovka838@gmail.com}

        \author{Sergey Rogovoy}
        \address{Lomonosov MSU, Faculty of Computational Mathematics and Cybernetics, Moscow, Russia, 119991}
        \email{rogovoyserg@gmail.com}
        
	\begin{abstract}
		We establish a version of the Landen's transformation for Weierstrass functions and invariants that is applicable to general lattices in complex plane. Using it we present an effective method for computing Weierstrass functions, their periods, and elliptic integral in Weierstrass form given Weierstrass invariants $g_2$ and $g_3$ of an elliptic curve. Similarly to the classical Landen's method our algorithm has quadratic rate of convergence.
		
	\end{abstract}
	\maketitle

\section{Introduction}
There are numerous applications of elliptic functions in various fields of mathematics and physics (see, e.g.~\cite{Akhiezer},~\cite{Lawden}, and references therein). For efficient computations with elliptic functions various methods are known. Traditionally, the most frequently used approaches include either the Landen method (or equivalent methods based on the arithmetic-geometric mean) for Jacobi functions~\cite{luther}, or summation of theta-series~\cite{Johansson}. It is noteworthy that Weierstrass functions, even though they are the most convenient for theoretical framework of elliptic functions and curves, are not usually considered as an effective computational tool. In this paper we present an approach related to the Landen method that provides effective computation of all the Weierstrass functions, elliptic integral in the Weierstrass form, and periods of an elliptic curve given the Weierstrass invariants $g_2, g_3$.

    The part of the method concerning computation of periods and the Abel map was already presented in the work~\cite{Cremona} in an equivalent form with the use of the complex version of arithmetic-geometric mean. Also a similar method for computation of the $\wp$-function was presented in~\cite[Sec.~4.3]{labrande}. In this work we present a unified approach to these computations, as well as the computation of other Weierstrass functions. It is noteworthy that the calculation of all Weierstrass functions simultaneously allows not only to provide a framework to all possible computations with elliptic functions, but also to solve the problem that emerges in adaptations of the Landen or AGM-type methods to theta-functions. Namely, for the theta function (and the Weierstrass $\sigma$ function as well) Landen's transformation only allows  to determine the squared value of the function. For example, in~\cite{labrandeg1} this problem is solved by using a low-accuracy approximation of the theta-function by its Fourier series, which allows to determine the sign. Such an approach leads to a significant increase in computational complexity, so the method shows its efficiency only during calculations with very high precision. For the Weierstrass $\sigma$ function, however, this difficulty can be overcome by computing simultaneously $\wp, \wp'$, and $\sigma^2$ at $z/2$ and using the duplication formula.

    Similarly to the Landen and AGM-based methods the method presented here has quadratic rate of convergence. For example, in practice it is usually sufficient to perform at most $5$ iterations of the Landen transformation in order to achieve the machine precision while computing with double precision floating point arithmetic. Moreover, due to the quadratic convergence, there is only mild increase in complexity with the use of the high precision arithmetic.
    Finally, we note that the essence of this method implies that the computations are stable for curves that are close to being degenerate. To demonstrate this we compute parameters of a conformal mapping problem studied in~\cite{Smirnov} for domains, such that the corresponding elliptic curve is near degeneration.

    The paper is organized as follows. In Section~\ref{sec:preliminaries} we recall the notation and several facts from the theory of elliptic functions. In particular, we are interested in the Weierstrass functions that are associated not only with lattices but also with subgroups of $\CC$ of rank $1$ or $0$, which are not usually covered in literature. The main results of this paper are collected in Section~\ref{sec:sublattices}. Namely, we derive the Landen transformation for the Weierstrass functions, describe the optimal way to choose a subgroup of index $2$ in a lattice, and analyze the rate of convergence of Weierstrass invariants under the iterations of the Landen transformation. Some of the results of Section~\ref{sec:sublattices} are known in literature, but it was decided to include their proofs for completeness and convenience. Section~\ref{sec:periods} contains the description of the method that computes periods and the Abel map (i.e. the elliptic integral in Weierstrass form). Since an equivalent version of this method was already studied in~\cite{Cremona}, we do not include analysis of the convergence. In Section~\ref{sec:computation} we describe the algorithm to compute Weierstrass functions and give a simple proof of the quadratic rate of convergence for the $\wp$ function. Finally, Section~\ref{sec:experiments} contains numerical experiments that show the quadratic convergence and an application to a conformal mapping problem.

    \section{Preliminaries}
\label{sec:preliminaries}

The letter $\Gamma$ will always denote a discrete additive subgroup in $\CC$. It is clear that such group is free and has rank not exceeding $2$. A subgroup $\Gamma$ is called a {\it{lattice}} if its rank is exactly $2$. Given $A \subset \CC$ we denote the smallest additive subgroup of $\CC$ containing $A$ by $\Span A$. That is, $\Span A$ is the integer linear span of $A$. As usual, if $\Gamma$ is a lattice and $f$ is a meromorphic function on $\CC$, such that all elements of $\Gamma$ are periods of $f$, we say that $f$ is $\Gamma$-{\it{elliptic}}.

Given a {\it{discrete subgroup}} $\Gamma \subset \CC$ we define $\wp, \zeta$ and $\sigma$ as
$$\wp(z;\Gamma) =
	\frac{1}{z^2} + \sum_{u \in \Gamma \setminus \{0\}}\left(\frac{1}{(z - u)^2} - \frac{1}{u^2} \right),
$$
$$\zeta(z;\Gamma) = \frac{1}{z} + \sum_{u \in \Gamma \setminus \{0\}}\left(\frac{1}{z - u} + \frac{1}{u} + \frac{z}{u^2} \right),
$$
$$\sigma(z;\Gamma) = z\prod_{u \in \Gamma \setminus \{0\}} \left(1 - \frac{z}{u}\right)e^{\frac{z}{u} + \frac{z^2}{2u^2}}.
$$

Moreover, functions
\begin{equation}\label{gw}
    g_2(\Gamma) = \sum_{u \in \Gamma \setminus \{0\}}\frac{60}{u^4},\;\; g_3(\Gamma) = \sum_{u \in \Gamma \setminus \{0\}}\frac{140}{u^6}
\end{equation}
establish a bijective mapping from the set of all discrete additive subgroups of $\CC$ to $\CC^2$. We note that $\Gamma \subset \CC$ is a lattice if and only if
\begin{equation}\label{eq_DeltaDef}
    \Delta(\Gamma) = g_2(\Gamma)^3 - 27 g_3(\Gamma)^2
\end{equation}
does not vanish. The function $\sigma(z; \Gamma)$ is an entire function of the variables $(z, g_2, g_3)$ (see, e.g.,~\cite{Weier}).

Finally, we recall that $(g_2, g_3) \in \CC^2$ define a curve, whose affine part is given by the equation
$$y^2 = 4x^3 - g_2 x - g_3.$$
If $\Gamma$ is a lattice, then the curve $S(\Gamma)$, that corresponds to $(g_2(\Gamma), g_3(\Gamma))$, is a nonsingular (elliptic) curve isomorphic to $\CC / \Gamma$. The equivalence of the foregoing curves is established by the Abel map $\mathcal A_\Gamma: S(\Gamma) \rightarrow \CC/\Gamma$, given by the {\it{elliptic integral}}
$$
    \mathcal A_\Gamma (x,y) = \int_{(\infty, \infty)}^{(x,y)} dx/y \quad \mathrm{mod}\; \Gamma.
$$
The inverse mapping can be computed as $(x,y) = (\wp(z; \Gamma), \wp'(z; \Gamma))$, where $z + \Gamma = \mathcal A_\Gamma(x,y)$. The foregoing formulae can be used in the case when $\Gamma$ is an arbitrary discrete subgroup of $\CC$,  not necessary a lattice. Then $S(\Gamma)$ has a unique singular point $s$, and $\mathcal A_\Gamma$ is an equivalence between $S(\Gamma)\setminus \{s\}$ and $\CC / \Gamma$. The formula for the inverse remains unchanged.

We will usually use the roots of the polynomial $4x^3 - g_2 x - g_3$ instead of coefficients $g_2, g_3$. It is clear that the roots $e_1, e_2, e_3$ satisfy
\begin{equation}\label{eq_InvFromRoots}
    g_2 = 2(e_1^2 + e_2^2 + e_3^2), \;\; g_3 = 4 e_1 e_2 e_3, \;\; g_2^3 - 27 g_3^2 = 16(e_1 - e_2)^2(e_2 - e_3)^2(e_1 - e_3)^2.
\end{equation}
If $\Gamma$ is a lattice, then the roots $e_1,e_2,e_3$ of the corresponding polynomial (with coefficients $g_2(\Gamma), g_3(\Gamma)$) are simple and $\{e_1,e_2,e_3\} = \{\wp(z): z \in (\Gamma/2) \setminus \Gamma\}$.

Finally, if $\Gamma \subset \CC$ is a discrete subgroup but not a lattice, then the corresponding Weierstrass functions can be explicitly found in elementary functions (see, e.g.~\cite[p.~201, Table~VII]{Akhiezer}). In particular, if $\rank(\Gamma) = 1$, and $\omega$ is a generating element of $\Gamma$ (that is $\Gamma = \omega \ZZ$), then

    $$
    \wp(z;\Gamma) = \frac{\pi^2}{\omega^2} \left( \frac{1}{\sin ( \frac{\pi z}{\omega})^2} - \frac{1}{3} \right),$$
 %   $$\wp'(z;\Gamma) = \frac{-2  \pi^3  \cos(\pi  \frac{z}{\omega})}{\omega^3 \sin\frac{\pi z}{\omega})^3} ,$$
    $$
    \zeta(z;\Gamma) = \frac{\pi^2 z}{ 3 \omega^2} + \frac{\pi}{\omega}  \cot\left(\pi  \frac{z}{\omega}\right),
    $$
    $$
    \sigma(z;\Gamma) = \frac{\omega}{\pi}  \exp\left(  \frac{\pi^2 z^2}{6 \omega^2}\right) \sin\left(\pi \frac{z}{\omega}\right),
    $$
    $$
    \mathcal A_{\Gamma}(x,y) = - \frac{\omega}{\pi}\arctan\left(\frac{6\pi \omega^2 x + 2\pi^3}{3\omega^3 y}\right) + \Gamma
    $$
    Moreover,
    $$ g_2(\Gamma) = \frac{4\pi^4}{3\omega^4},\;\; g_3(\Gamma) = \frac{8\pi^6}{27\omega^6},$$
    and the roots of the polynomial $4x^3 - g_2(\Gamma)x - g_3$ can be calculated as
    \begin{equation}\label{eq_limRoots}
        e_1 = \wp\left(\frac{\omega}{2}\right) = \frac{2\pi^2}{3\omega^2}\,,  \quad e_2 = e_3 = -\frac{e_1}{2}\,.
    \end{equation}
    Finally, if $\rank(\Gamma) = 0$ (that is, $\Gamma = \{0\}$), we get
    $$\wp(z;\Gamma) = \frac{1}{z^2}\,,\;\; \zeta(z;\Gamma) = \frac{1}{z}\,,\;\;\sigma(z;\Gamma) = z,\;\; \mathcal A_{\Gamma}(x,y) = -2\,\frac{x}{y}\,,$$
    $$g_2(\Gamma) = g_3(\Gamma) = 0.$$

\section{Subgroups of index $2$ in lattices}
\label{sec:sublattices}
Throughout this section $\Gamma \subset \CC$ denotes a lattice. Recall that the {\it{index}} $[\Gamma:G]$ of a subgroup $G \subset \Gamma$ is the number of elements in the quotient group $\Gamma / G$.

\begin{proposition}\label{p31}
    Let $\Gamma' \subset \Gamma$ be a subgroup of index $2$. Then the following statements hold.
    \begin{enumerate}[label=(\roman*)]
        \item\label{p31i} $\Gamma'$ is a lattice.
        \item\label{p31ii} $2\Gamma \subset \Gamma' \subset \Gamma$ and $[\Gamma':2\Gamma] = 2$.
        \item\label{p31iii} Let $\omega \in \Gamma' \setminus 2\Gamma$. Then $\Gamma' = \Span (2\Gamma \cup \{\omega\})$.
    \end{enumerate}
\end{proposition}
\begin{proof}
    The statement \ref{p31i} is clear. To prove \ref{p31ii} note that $2x = 0$ for all $x \in \Gamma / \Gamma'$. Therefore, $2x + \Gamma' = \Gamma'$ for all $x \in \Gamma$ which implies $2x \in \Gamma'$. The equality $[\Gamma':2\Gamma] = 2$ is easily derived using $[\Gamma:\Gamma'] = 2$ and $[\Gamma: 2\Gamma] = 4$.

    Now we prove \ref{p31iii}. Clearly, $G = \Span (2\Gamma \cup \{\omega\}) \subset \Gamma'$. Since $2 \Gamma \subset G \subset \Gamma$, the number $[\Gamma : G]$ can be only equal to  $1$, $2$, or $4$. Since $G$ is strictly greater than $2 \Gamma$, its index cannot be equal to $4$. On the other hand, $G$ is contained in $\Gamma'$, so $G \ne \Gamma$. Thus, $[\Gamma : G] = 2$. It remains to note $\Gamma'$ has the same index in $\Gamma$ as $G$ and the equality $[\Gamma:G] = [\Gamma:\Gamma'] [\Gamma':G]$ implies that $\Gamma'/G$ is a trivial group.
\end{proof}
\begin{corollary}\label{c31}
    Let $\omega \in \Gamma \setminus 2\Gamma$. Then there exists a unique subgroup $\Gamma' \subset \Gamma$ of index $2$, that contains $\omega$. Moreover, $\Gamma' = \Span(2\Gamma \cup \{\omega\})$.
\end{corollary}
\begin{proof}
    The formula for $\Gamma'$ (if it exists) is contained in Proposition~\ref{p31}~\ref{p31iii}. The uniqueness follows. It remains to show that $G = \Span(2\Gamma \cup \{\omega\})$ has index $2$. But this is clear, since $2\Gamma$ has index $4$ and $G$ is strictly between $2\Gamma$ and $\Gamma$.
\end{proof}
\begin{corollary}
    A lattice $\Gamma$ has exactly three subgroups of index $2$. More precisely, let $\omega_1, \omega_2$ be a basis in $\Gamma$. Then the groups $$\Span(2\Gamma\cup \{\omega_1\}),\;\; \Span(2\Gamma\cup \{\omega_2\}),\;\; \Span(2\Gamma\cup \{\omega_1+\omega_2\})$$ are the only subgroups of index $2$ in $\Gamma$.
\end{corollary}
\begin{proof}
    It is clear that these groups are distinct and all have index $2$. It remains to prove that there are no other such groups. If $\Gamma' \subset \Gamma$ is a subgroup of index $2$ that is distinct from them, then $\omega_1, \omega_2, \omega_1 + \omega_2 \notin \Gamma'$ by Proposition~\ref{p31}~\ref{p31iii}. It follows that $\omega_1 + \Gamma', \omega_2 + \Gamma'$ are distinct nonzero elements in $\Gamma / \Gamma'$, which contradicts the assumption $[\Gamma, \Gamma'] = 2$.
\end{proof}

Now we fix a subgroup $\hat{\Gamma} \subset \Gamma$ of index $2$ and derive the Landen transformation of the Weierstrass functions.
\begin{proposition}\label{p32}
     For brevity we denote by $\wp(z), \zeta(z), \sigma(z)$ the Weierstrass functions corresponding to $\Gamma$, and by $\hat{\wp}(z), \hat{\zeta}(z), \hat{\sigma}(z)$ the Weierstrass functions corresponding to $\hat{\Gamma}$. Let us fix $\omega_1 \in \hat{\Gamma} \setminus 2\Gamma$ and $\omega_2 \in \Gamma \setminus \hat{\Gamma}$. Also let $e_1, e_2, e_3$ be (distinct) roots of the polynomial $4x^3 - g_2(\Gamma) x - g_3(\Gamma)$, and let $e_1 = \wp(\omega_1/2)$. Finally, let $\hat{e}_1, \hat{e}_2, \hat{e}_3$ denote the roots of $4x^3 - g_2(\hat{\Gamma}) x - g_3(\hat{\Gamma})$, and let $\hat{e}_1 = \hat{\wp}(\omega_2)$. Then the following statements hold.
     \begin{enumerate}[label=(\roman*)]
        \item\label{p32i} For the Weierstrass functions we have relations
        \begin{equation}\label{eq_PP2}\
            \wp(z) = \hat{\wp}(z) + \frac{(\hat{e}_2 - \hat{e}_1)(\hat{e}_3 - \hat{e}_1)}{\hat{\wp}(z) - \hat{e}_1},
    \end{equation}

        \begin{equation}\label{eq_PzPz}\
            \wp'(z) = \hat{\wp}'(z)\left(1 - \frac{(\hat{e}_2 - \hat{e}_1)(\hat{e}_3 - \hat{e}_1)}{(\hat{\wp}(z) - \hat{e}_1)^2}\right),
    \end{equation}

        \begin{equation}\label{eq_zeta_zeta}
     \zeta(z) = 2\hat{\zeta}(z)+\frac{1}{2}\frac{\hat{\wp'}(z)}{\hat{\wp}(z) - \hat{e_1}} + \hat{e}_1 z,
     \end{equation}

        \begin{equation}\label{eq_sigma_sigma}
     \sigma^2(z) = \exp(\hat{e}_1 z^2) (\hat{\wp}(z) - \hat{e}_1) \hat{\sigma}^4(z).
     \end{equation}
        \item\label{p32ii} The roots satisfy
        \begin{equation}\label{eq_ee}
            \hat{e}_1 = -\frac{e_1}{2}, \;\;
            16(\hat{e}_2 - \hat{e}_1)(\hat{e}_3 - \hat{e}_1) = (e_2 - e_3)^2.
        \end{equation}
        \item\label{p32iii} The Weierstrass invariants satisfy
        \begin{equation}\label{eq_g2}
            g_2(\hat{\Gamma}) = \frac{3e^2_1}{4} + (e_1 - e_2)(e_1 - e_3),
        \end{equation}
        \begin{equation}\label{eq_g3}
            g_3(\hat{\Gamma}) = -\frac{e^3_1}{8} + \frac{e_1}{2}(e_1 - e_2)(e_1 - e_3),
        \end{equation}
        \begin{equation}\label{eq_delta}
            \Delta(\hat{\Gamma}) = \frac{1}{16}(e_1 - e_2)(e_1 - e_3)(e_2 - e_3)^4.
        \end{equation}
     \end{enumerate}
\end{proposition}
\begin{proof}
    At first note that the following relation between $\wp$ and $\hat{\wp}$ holds:
    \begin{equation}\label{eq_PP1}\
    \wp(z) = \hat{\wp}(z) + \hat{\wp}(z+\omega_2) - \hat{\wp}(\omega_2).
    \end{equation}
    It can be easily verified by checking that both the sides of the equality are $\Gamma$-elliptic and have common poles with coinciding non-positive parts of Laurent series at all poles. After that~\eqref{eq_PP1} can be rewritten as~\eqref{eq_PP2} using the equality~\cite[p.~200, Table~VI]{Akhiezer}
    \begin{equation}\label{eq_P-e2}
        \hat{\wp}(z+\omega_2) - \hat{\wp}(\omega_2) = \frac{(\hat{e}_2 - \hat{e}_1)(\hat{e}_3 - \hat{e}_1)}{\hat{\wp}(z) - \hat{e}_1}.
    \end{equation}

    The relation~\eqref{eq_PzPz} is obtained from~\eqref{eq_PP2} by differentiation. In order to prove~\eqref{eq_zeta_zeta} we integrate~\eqref{eq_PP1} and get
    \begin{equation}\label{eq_zeta_zeta_1}
        \zeta(z) = \hat{\zeta}(z) + \hat{\zeta}(z+\omega_2) + \hat{e}_1 z - \hat{\zeta}(\omega_2).
    \end{equation}
    Now the addition formula for $\hat{\zeta}$ (see \cite[Eq.~18.4.3]{AbramovitzStegun}) applied to~\eqref{eq_zeta_zeta_1} easily implies~\eqref{eq_zeta_zeta}. Finally, by integrating~\eqref{eq_zeta_zeta_1} once again we get
    $$
        \sigma(z) = \hat{\sigma}(z)\hat{\sigma}(z+\omega_2) \exp(\hat{e}_1 z^2 / 2 - \hat{\zeta}(\omega_2)z) / \hat{\sigma}(\omega_2).
    $$
    Now the equality~\eqref{eq_sigma_sigma} follows by squaring both the sides and applying addition formula~\cite[Eq.~18.4.4]{AbramovitzStegun}. That is,~\ref{p32i} is proved.

    In order to prove~\ref{p32ii} we substitute $z = \omega_1/2$ into~\eqref{eq_PP2}. It is clear that $\hat{\wp}(z)$ equals either $\hat{e}_2$, or $\hat{e}_3$. Both the cases lead to the equality $e_1 = \hat{e}_2 + \hat{e}_3 - \hat{e}_1 = -2 \hat{e}_1$. To prove the second relation from~\eqref{eq_ee} we substitute $z = \omega_2/2$ into~\eqref{eq_PP2} and get
    $$
        \wp\left(\frac{\omega_2}{2}\right) + \frac{e_1}{2} = \hat{\wp}\left(\frac{\omega_2}{2}\right) - \hat{e}_1 + \frac{(\hat{e}_2 - \hat{e}_1)(\hat{e}_3 - \hat{e}_1)}{\hat{\wp}(\omega_2/2) - \hat{e}_1}.
    $$
    It is clear that $\wp(\omega_2/2)$ equals either $e_2$, or $e_3$. It easily follows that $\wp(\omega_2/2) + e_1/2 = \pm(e_2 - e_3)/2$. Thus,
    $$
        \frac{(e_2 - e_3)^2}{4} = (\hat{\wp}\left(\frac{\omega_2}{2}\right) - \hat{e}_1)^2 + 2(\hat{e}_2 - \hat{e}_1)(\hat{e}_3 - \hat{e}_1) + \frac{(\hat{e}_2 - \hat{e}_1)^2(\hat{e}_3 - \hat{e}_1)^2}{(\hat{\wp}(\omega_2/2) - \hat{e}_1)^2}.
    $$
    Now note that substituting $z = \omega_2/2$ into~\eqref{eq_P-e2} implies $(\hat{\wp}(\omega_2/2) - \hat{e}_1)^2 = (\hat{e}_2 - \hat{e}_1)(\hat{e}_3 - \hat{e}_1)$. Finally, we get the second relation from~\eqref{eq_ee}.

    The statement~\ref{p32iii} is an elementary corollary of~\ref{p32ii} and the formulae that express $g_2(\hat{\Gamma}), g_3(\hat{\Gamma})$, and $\Delta(\hat{\Gamma})$ in terms of $\hat{e}_1, \hat{e}_2, \hat{e}_3$.
\end{proof}
\begin{remark}
    It can be easily verified that~\eqref{eq_ee}, along with the trivial equality $\hat{e}_1 + \hat{e}_2 + \hat{e}_3 = 0$, determine the numbers $\hat{e}_1, \hat{e}_2, \hat{e}_3$ completely up to the transposition of $\hat{e}_2$ and $\hat{e}_3$. Indeed, let $16r^2 = 9e_1^2 - (e_2 - e_3)^2 = 4(e_1 - e_2)(e_1 - e_3)$. Then $\hat{e}_{2,3} = e_1/4 \pm r$ are the solutions of~\eqref{eq_ee}. There are only two solutions that correspond to the choice of $r$ (or, equivalently, the transposition of $\hat{e}_2$ and $\hat{e}_3$).
\end{remark}
Proposition~\ref{p32} contains all necessary information to formulate a Landen-type method. However, it is not clear how to choose one of three subgroups of index $2$ in a given lattice. We address this problem below.
\begin{definition}
    We say that complex numbers $\omega_1, \omega_2 \in \Gamma$ constitute a reduced basis of $\Gamma$ if $$|\omega_1| = \inf \{ |w|:w \in \Gamma\setminus \{0\}\}, \;\;
	 |\omega_2| = \inf \{|w|:  w \in  \Gamma \setminus \omega_1\ZZ \}. $$
\end{definition}
It is clear that a reduced basis in $\Gamma$ indeed is a basis. In addition we note that a reduced basis in $\Gamma$ always exists. The main property of a reduced basis is given in the following theorem.
\begin{theorem}\label{tRoots}
        Let $\omega_1$,  $\omega_2$ be a reduced basis in $\Gamma$ and let
	$$e_1 = \wp\left(\frac{\omega_1}{2}\,; \Gamma\right),\;\; e_2 = \wp\left(\frac{\omega_2}{2}\,; \Gamma\right),\;\;  e_3 = \wp\left(\frac{\omega_1 + \omega_2}{2}; \Gamma\right). $$
	Then the following statements hold.
 \begin{enumerate}[label=(\roman*)]
    \item\label{tRootsi} $|e_2 - e_3| \leq |e_1 - e_3| \leq |e_1 - e_2|$.
    \item\label{tRootsii} $|e_2 - e_3| = |e_1 - e_3|$ if and only if $|\omega_1| = |\omega_2|$.
    \item\label{tRootsiii} $|e_1 - e_3| = |e_1 - e_2|$ if and only if either $|\omega_1 + \omega_2| = |\omega_2|$, or $|\omega_1 - \omega_2| = |\omega_2|$.
 \end{enumerate}
\end{theorem}
We will prove several lemmas prior to dealing with this theorem. Let $\Phi$ denote the set $\{z \in \CC: |z| > 1, |\mathrm{Re}\: z| < 1/2\}$. The following lemma is trivial and we omit its proof.
\begin{lemma}\label{l31}
    Let $\mathrm {Im}\:\tau > 0$ and $\Gamma = \Span \{1, \tau\}$. Then $\tau \in \overline{\Phi}$ if and only if $1,\tau$ is a reduced basis in~$\Gamma$.
\end{lemma}
It is clear that in Theorem~\ref{tRoots} we can assume without loss of generality that $\omega_1 = 1$ and $\mathrm{Im}\:\omega_2 > 0$. Thus, Lemma~\ref{l31} implies that $\omega_2 \in \overline{\Phi}$. Now let $\Gamma(\tau)$ denote $\Span\{1, \tau\}$, $\mathrm{Im}\: \tau > 0$. Then we introduce $e_1(\tau), e_2(\tau), e_3(\tau)$ via formulae
$$e_1(\tau) = \wp(1/2;\Gamma(\tau)),\;\; e_2(\tau) = \wp(\tau/2;\Gamma(\tau)),\;\;e_3(\tau) = \wp((\tau + 1)/2; \Gamma(\tau)).$$
Also let
$$f(\tau) = \frac{e_1(\tau) - e_3(\tau)}{e_1(\tau) - e_2(\tau)}, \; \; \; g(\tau) = \frac{e_2(\tau) - e_3(\tau)}{e_1(\tau) - e_3(\tau)}.$$
It is clear that $f$ and $g$ are holomorphic in the upper half-plane. Moreover, statement~\ref{tRootsi} of Theorem~\ref{tRoots} can be now reformulated as: $|f(\tau)| \le 1$ and $|g(\tau)| \le 1$ for $\tau \in \overline{\Phi}$. Analogously, the statements~\ref{tRootsii} and~\ref{tRootsiii} can be reformulated in terms of $f$ and $g$.

\begin{remark}
	The function $f$ is a well-known example of a modular function (see, e.g. \cite[Chapter~III]{AhlforsQuasiConf}), and $g$ is related to $f$ by a Mobius transformation $g = 1 - 1/f$. It is known that the function $f$ maps the region $\{z \in \CC: 0 < \mathrm{Re} z < 1,\;\; |z - 1/2| > 1/2\}$ conformally onto the lower half-plane. On the domain $\Phi$ the functions $f$ and $g$ also turn out to be conformal mappings. Using Lemma~\ref{l34} below it is possible to find the images $f(\Phi)$ and $g(\Phi)$. Indeed, it is clear that $|g(\tau)| = 1$ if and only if $\mathrm{Re}f(\tau) = 1/2$ and $|f(\tau)| = 1$ if and only if $|g(\tau) - 1| = 1$. Thus, $f(\Phi) = \{z \in \CC: |z| < 1,\;\; \mathrm{Re}z > 1/2\}$ and $g(\Phi) = \{z \in \CC: |z| < 1,\;\; |z - 1| > 1\}$.
\end{remark}

\begin{lemma}\label{l32}
    As $\mathrm{Im}\: \tau \rightarrow +\infty$ the functions $f(\tau)$ and $g(\tau)$ converge to $1$ and $0$ respectively uniformly with respect to $\mathrm{Re}\: \tau$.
\end{lemma}
\begin{proof}
We use the series expansions of $e_j(\tau)$ (see \cite[p.~204, Table~X]{Akhiezer}). Let $\eta(\tau) = 2\zeta(1/2; \Gamma(\tau)).$ Then the following relations hold:
$$
e_1(\tau) = -\eta(\tau) + \pi^2\left(\frac{1}{4} + 2 \sum_{k = 1}^{\infty}\frac{\exp{2\pi i \tau k}}{(1+\exp{2\pi i \tau k})^2} \right),
$$
$$
e_2(\tau) = -\eta(\tau) + 2\pi^2 \sum_{k = 1}^{\infty}\frac{\exp{2\pi i \tau (2k-1)}}{(1+\exp{\pi i \tau (2k-1)})^2},
$$
$$
e_3(\tau) = -\eta(\tau) - 2\pi^2 \sum_{k = 1}^{\infty}\frac{\exp{2\pi i \tau (2k-1)}}{(1+\exp{\pi i \tau (2k-1)})^2}.
$$
It is clear from these expansions that $|e_2(\tau) - e_3(\tau)| \rightarrow 0$, $|e_1(\tau) - e_2(\tau)|, |e_1(\tau) - e_3(\tau)| \rightarrow \pi^2/4$ uniformly with respect to $\mathrm{Re}\: \tau$ as $\mathrm{Im}\:\tau \rightarrow +\infty$. %The lemma follows.
\end{proof}

Now we consider the behaviour of $f$ and $g$ at the boundary of $\Phi$. The main tool is the following lemma.
\begin{lemma}\label{l33}
Let $\Gamma$ be self-conjugate (that is, $w \in \Gamma$ if and only if $\bar{w} \in \Gamma$), and let $\omega_1, \omega_2 \in \Gamma$ denote a basis such that $\omega_2 = \overline{\omega_1}$. Let $e_1 = \wp(\omega_1/2; \Gamma), e_2 = \wp(\omega_2/2; \Gamma)$, and $e_3 = \wp((\omega_1+\omega_2)/2; \Gamma)$. Then $e_2 = \overline{e_1}$ and $e_3 \in \RR$. In particular, $|e_1 - e_3| = |e_2 - e_3|$.
\end{lemma}
\begin{proof}
    The statement elementary follows from the equality $\wp(z; \Gamma) = \overline{\wp(\bar{z}; \Gamma)}$.
\end{proof}
\begin{lemma}\label{l34}
    Assume that $\tau \in \partial \Phi$.
    \begin{enumerate}[label=(\roman*)]
        \item\label{l34i} If $|\tau| = 1$, then $|g(\tau)| = 1$ and $|f(\tau)| \le 1$. Moreover, if $|\mathrm{Re}\:\tau| \ne 1/2$, then $|f(\tau)| < 1$.
        \item\label{l34ii} If $|\mathrm{Re} \:\tau| = 1/2$, then $|f(\tau)| = 1$ and $|g(\tau)| \le 1$. Moreover, if $|\tau| > 1$, then $|g(\tau)| < 1$.
    \end{enumerate}
\end{lemma}
\begin{proof}
    At first we note that if $\tau \in \partial \Phi$ and $|g(\tau)| = |f(\tau)| = 1$, then either $\tau = \exp(i\pi/3)$, or $\tau = \exp(2i\pi/3)$. Indeed, equality $|g(\tau)| = |f(\tau)| = 1$ means that $|e_1(\tau) - e_3(\tau)| = |e_2(\tau) - e_3(\tau)| = |e_1(\tau) - e_2(\tau)|$. It is easy to verify that a lattice $\Gamma$ such that the roots of the corresponding polynomial form an equilateral triangle satisfies $\Gamma = \exp(i\pi/3)\Gamma$. Since we consider lattices of the form $\Gamma(\tau)$, $\tau \in \partial \Phi$ the statement easily follows.

    Now we prove~\ref{l34i}. It is clear that $\tau = \exp(i \alpha), \alpha\in [\pi/3, 2\pi/3]$. It is clear that $\hat{\Gamma} = \exp(-i\alpha/2)\Gamma(\tau)$ is self-conjugate and Lemma~\ref{l33} implies $|e_1 - e_3| = |e_2 - e_3|$, so $|g(\tau)| = 1$. Note that $|f(\tau(\alpha))|$ is continuous on the interval $[\pi/3, 2\pi/3]$ and by the statement above $|f(\tau(\alpha))| = 1$ only at the endpoints. Thus, either $|f(\tau(\alpha))| \ge 1$, or $|f(\tau(\alpha))| \le 1$ on the whole interval $[\pi/3, 2\pi/3]$. To check which one is true it is enough to consider an arbitrary point for which the computations can be done explicitly, for example, $\alpha = \pi/2$. In this case, clearly, $e_1(\tau) = -e_2(\tau)$ and $e_3(\tau) = 0$. Thus, $|f(\tau)| = 1/2 < 1$.

    Now we prove~\ref{l34ii}. Assume that $|\mathrm{Re}\:\tau| = 1/2$. It is easy to see that $\hat{\Gamma} = i \Gamma(\tau)$ is self-conjugate. Similarly to the previous statement we conclude that $|e_1 - e_3| = |e_1 - e_2|$ and $|f(\tau)| = 1$. Again, only one of the inequalities $|g(\tau)| \ge 1$ and $|g(\tau)| \le 1$ holds for all $\tau \in \partial \Phi$, $\mathrm{Re}\:\tau = 1/2$, since the function $|g(\tau)|$ is continuous and the set of $\tau$ under consideration is connected. In this case Lemma~\ref{l32} implies that the inequality $|g(\tau)| \le 1$ holds for $\tau \in \partial \Phi$, $\mathrm{Re}\:\tau = 1/2$. The other part of the boundary where $\mathrm{Re}\:\tau = -1/2$ is treated similarly.
\end{proof}

\begin{proof}[Proof of the Theorem~\ref{tRoots}]
    Note that Lemmas~\ref{l32} and~\ref{l34} combined with the maximum modulus principle imply $|f(\tau)| \le 1$ and $|g(\tau)| \le 1$ for $\tau \in \overline{\Phi}$. This proves~\ref{tRootsi}. Moreover $|g(\tau)| < 1$ and $|f(\tau)| < 1$, if $\tau \in \Phi$ (also by the maximum modulus principle). Using Lemma~\ref{l34}~\ref{l34i} we get that for $\tau \in \overline{\Phi}$ the equality $|f(\tau)| = 1$ is equivalent to $|\mathrm{Re}\:\tau| = 1/2$. Since in this case either $|\tau + 1| = |\tau|$, or $|\tau - 1| = |\tau|$, we get~\ref{tRootsii}. Finally, for $\tau \in \overline{\Phi}$ Lemma~\ref{l34}~\ref{l34ii} implies that $|g(\tau)| = 1$ if and only if $|\tau| = 1$. This completes the proof of~\ref{tRootsiii}.
\end{proof}
From now on we will call a triple of complex numbers $(e_1,e_2,e_3)$ {\it{properly ordered}} if $|e_2 - e_3| \leq |e_1 - e_3| \leq |e_1 - e_2|$. Usually a properly ordered triple of numbers cannot be reordered with preservation of properness. However, each possible proper order of the roots of polynomial $4x^3 - g_2(\Gamma)x - g_3(\Gamma)$ corresponds to a reduced basis in $\Gamma$.
\begin{corollary}\label{ctRoots}
    Let $(e_1,e_2,e_3)$ be a properly ordered triple of distinct roots of the polynomial $4x^3 - g_2(\Gamma)x - g_3(\Gamma)$. Then there exists a reduced basis $\omega_1, \omega_2$ in $\Gamma$ such that
	$$e_1 = \wp\left(\frac{\omega_1}{2}; \Gamma\right),\;\; e_2 = \wp\left(\frac{\omega_2}{2}; \Gamma\right),\;\;  e_3 = \wp\left(\frac{\omega_1 + \omega_2}{2}; \Gamma\right). $$
\end{corollary}
\begin{proof}
Consider arbitrary reduced basis $\phi_1,\phi_2$ in $\Gamma$. Let 
$$
f_1 = \wp\left(\frac{\phi_1}{2}; \Gamma\right),\;\; f_2 = \wp\left(\frac{\phi_2}{2}; \Gamma\right),\;\;  f_3 = \wp\left(\frac{\phi_1 + \phi_2}{2}; \Gamma\right). 
$$
Obviously $(f_1,f_2,f_3)$ is a shuffle of numbers $(e_1,e_2,e_3)$, and by Theorem~\ref{tRoots} it is properly ordered. Now it is clear that $|f_2 - f_3| = |e_2 - e_3|$, $|f_1 - f_3| = |e_1 - e_3|$, and $|f_1 - f_2| = |e_1 - e_2|$.

    If $e_1 = f_3$, then $e_1,e_2,e_3$ form an equilateral triangle and this case is handled trivially, since $\Gamma$ admits the basis $\omega_1, \omega_2$, where all three periods $\omega_1, \omega_2$, and $\omega_1 + \omega_2$ have the same (minimal) absolute value.

    Now assume that $e_1 = f_2$. Then is it clear that $|f_2 - f_3| = |f_1 - f_3|$ and Theorem~\ref{tRoots}~\ref{tRootsii} implies that $|\phi_1| = |\phi_2|$. Now, since $\phi_1$ and $\phi_2$ are interchangeable, we can only consider  the case $e_1 = f_1$. If $e_2 = f_2$, then the basis $\phi_1, \phi_2$ satisfies the requirements. Finally, we consider the case $e_2 = f_3$, so we get $|f_1 - f_3| = |f_1 - f_2|$ and by Theorem~\ref{tRoots}~\ref{tRootsiii} we get either $|\phi_2 + \phi_1| = |\phi_2|$, or $|\phi_2 - \phi_1| = |\phi_2|$. Now the required basis is $\omega_1 = \phi_1$ and $\omega_2 = \phi_2 \pm \phi_1$.
\end{proof}

Now we can conclude that if we identify different choices of basis that are obtained by changing the sign of vectors we have as many reduced bases in $\Gamma$ as there exists ways to properly order the roots of the polynomial $4x^3 - g_2(\Gamma)x - g_3(\Gamma)$. Moreover, in different reduced bases of the same lattice the absolute values of basis vectors are the same.

\begin{definition}
    We say that a subgroup $\Gamma'$ of index $2$ in $\Gamma$ is optimal if $|\Delta(\Gamma')| \le |\Delta(G)|$ for all subgroups $G \subset \Gamma$ of index $2$.
\end{definition}
 \begin{proposition}\label{p33}
    Let $\Gamma'$ be a subgroup of index $2$ in $\Gamma$. Then the following statements hold.
    \begin{enumerate}[label=(\roman*)]
        \item\label{p33i} Let $e_1,e_2,e_3$ be (distinct) roots of the polynomial  $4 x^3 - g_2 x - g_3$, and let $e_1 = \wp(\omega/2; \Gamma)$. Then $\Gamma'$ is optimal if and only if $|e_2 - e_3| \le \min\{|e_1 - e_2|, |e_1 - e_3|\}$.
        \item\label{p33ii} $\Gamma'$ is optimal if and only if there is $\omega \in \Gamma$ such that $\omega \in \Gamma'$ and $|\omega| = \inf \{|w|: w \in \Gamma \setminus \{0\}\}$.
    \end{enumerate}
 \end{proposition}
\begin{proof}
    The statement~\ref{p33i} is a simple consequence of~\eqref{eq_DeltaDef},~\eqref{eq_InvFromRoots}, and~\eqref{eq_delta}. Now~\ref{p33ii} is easily derived from~\ref{p33i} and Theorem~\ref{tRoots} alongside with its conversion Corollary~\ref{ctRoots}.

    %Now we prove~\ref{p33ii}. Now assume that there exists $\omega_1 \in \Gamma$ such that $\omega_1 \in \Gamma'$ and $|\omega_1| = \inf \{|w|: w \in \Gamma \setminus \{0\}\}$. Choose $\omega_2 \in \Gamma$ such that $|\omega_2| = \inf \{|w|: w \in \Gamma \setminus \omega\ZZ\}$, so $(\omega_1, \omega_2)$ is a basis of $\Gamma$. Now let $$e_1 = \wp\left(\frac{\omega_1}{2}; \Gamma\right),\;\; e_2 = \wp\left(\frac{\omega_2}{2}; \Gamma\right),\;\;  e_3 = \wp\left(\frac{\omega_1 + \omega_2}{2}; \Gamma\right). $$ Theorem~\ref{tRoots} implies that $|e_2 - e_3| \le \min\{|e_1 - e_2|, |e_1 - e_3|\}$, and, therefore,~\ref{p33i} implies that $\Gamma'$ is optimal.

    %Now assume that $\Gamma'$ is optimal. Again let $(\omega_1, \omega_2)$ denote a basis of $\Gamma$, satisfying the conditions of Theorem~\ref{tRoots} and consider roots $e_1, e_2, e_3$ as before. If $\omega_1 \in \Gamma'$, then~\ref{p33ii} is satisfied. Assume that $\omega_2 \in \Gamma'$. Then the statement~\ref{p33i} and Theorem~\ref{tRoots} imply that $|e_2 - e_3| = |e_1 - e_3|$ and $|\omega_1| = |\omega_2|$. So, statement~\ref{p33ii} is fulfilled. It remains to consider the case $\omega_1 + \omega_2 \in \Gamma'$. But then $|e_2 - e_3| = |e_1 - e_3| = |e_1 - e_2|$, and, as in the proof of Lemma~\ref{l34} it is clear that all subgroups of index $2$ in $\Gamma$ contain a vector of minimal length, as all such subgroups are obtained from each other via multiplying by $\exp(i\pi/3)$ or $\exp(2i\pi/3)$.
\end{proof}
\begin{corollary}\label{cp33}
    Let $\Gamma'$ be an optimal subgroup of index $2$ in $\Gamma$. Then $\Gamma'$ has only one optimal subgroup of index $2$,
\end{corollary}
\begin{proof}
    By Proposition~\ref{p33}~\ref{p33ii} there is $\omega_1 \in \Gamma'$ such that $|\omega_1| = \inf \{|w|: w \in \Gamma \setminus \{0\}\}$. If $\Gamma'$ contains two distinct optimal subgroups of index $2$, then there exists $\omega_2 \in \Gamma'$ such that $|\omega_2| = |\omega_1|$ and $\omega_2 \notin \omega_1 \ZZ$. It is clear that $(\omega_1, \omega_2)$ is a reduced basis in $\Gamma$ (hence, a basis), so $\Gamma' = \Gamma$. We arrived at a contradiction.
\end{proof}

To conclude this section we consider a sequence of subgroups of $\Gamma$ where each one of them has index $2$ and is optimal in the previous one. We analyse the convergence of the Weierstrass invariants corresponding to these lattices.

\begin{definition}
    We say that a sequence $\{a_n \in \CC\}_{n \in \NN}$ converges quadratically fast if it converges to some $a \in \CC$ and there exist $C, q > 0$ such that $|a_n - a| \le C\exp(-q2^n)$ for all $n \in \NN$.
\end{definition}

It is easy to verify that a sequence $\{a_n\}_{n \in \NN}$ that converges to $0$ and satisfies $|a_{n+1}| \le C |a_n|^2$ for large $n$ and some constant $C \ge 0$ converges to $0$ quadratically fast. The mentioned condition is the main source of sequences that converge quadratically fast. However, this condition does not behave well under certain operations (e.g., the sum of two sequences that satisfy this property may no longer be of the same type), and that is the reason why we prefer the definition above. Finally, we note that a sequence $\{a_n\}_{n \in \NN}$ converges quadratically fast if and only if the sequence $\{a_{n+1} - a_n\}_{n \in \NN}$ converges to $0$ quadratically fast.

\begin{lemma}\label{l35}
    Let $\Gamma_0 \supset \Gamma_1 \supset \Gamma_2 \supset \dots$ be a sequence of discrete subgroups of $\CC$. Let $G = \bigcap_{n = 0}^\infty \Gamma_n$. Then the following statements hold.
    \begin{enumerate}[label=(\roman*)]
        \item\label{l35i} $G$ is a lattice only if the sequence $\Gamma_n$ stabilizes.
        \item\label{l35ii} $g_2(\Gamma_n) \rightarrow g_2(G)$ and $g_3(\Gamma_n) \rightarrow g_3(G)$.
    \end{enumerate}
\end{lemma}
\begin{proof}
Assume that $G$ is a lattice. In this case the quotient group $\Gamma_0 / G$ is finite. Now it is clear that the sequence $\Gamma_n / G$ stabilises since it is a monotone sequence of subsets of a finite set. It easily follows that $\Gamma_n$ itself stabilises. The statement~\ref{l35i} is proved.

The statement~\ref{l35ii} trivially follows from equalities~\eqref{gw}.
\end{proof}

 \begin{lemma}\label{l36}
    Consider a sequence $\Gamma_0 \supset \Gamma_1 \supset \Gamma_2 \supset \dots$, where $\Gamma_i$ is optimal and has index $2$ in $\Gamma_{i-1}$. Let for each $n = 0,1,2\dots$ the triple $(e_1^{(n)}, e_2^{(n)}, e_3^{(n)})$ consists of the distinct roots of $4x^3 - g_2(\Gamma_n) x - g_3(\Gamma_n)$ such that  $e_1^{(n)} = \wp(\omega_n/2; \Gamma_n)$, where $\omega_n \in \Gamma_{n+1} \setminus 2\Gamma_n$. Then the following statement hold.
    \begin{enumerate}[label=(\roman*)]
        \item\label{l36i} $\bigcap_{n \in \NN} \Gamma_n$ is a discrete subgroup in $\CC$ of rank $1$, and its generating element $\phi$ satisfies $|\phi| = \inf \{ |w|: w \in \Gamma_0 \setminus \{0\}\}$.
        \item\label{l36ii} The sequence $\{\Delta(\Gamma_n)\}_{n \in \NN}$ converges to $0$ quadratically fast.
        \item\label{l36iii} The sequence $\{e_2^{(n)} - e_3^{(n)}\}_{n \in \NN}$ converges to $0$ quadratically fast.
    \end{enumerate}
 \end{lemma}
 \begin{proof}
Lemma~\ref{l35} implies that $G = \bigcap_{n \in \NN} \Gamma_n$ is not a lattice, so $\dim G \le 1$. Now let $\phi \in \Gamma_1$ satisfy $|\phi| = \inf\{|w|: w \in \Gamma_1 \setminus \{0\}\}$. Corollary~\ref{cp33} implies that $\phi$ is unique in $\Gamma_1$ up to the choice of sign. Now it is clear that $\phi \in \Gamma_n$ for all $n \in \NN$, so $\phi \in G$. So we get $\dim G = 1$ and $G = \phi \ZZ$. It remains to note that $\phi$ also satisfies $|\phi| = \inf \{ |w|: w \in \Gamma_0 \setminus \{0\}\}$ by Proposition~\ref{p33}~\ref{p33ii}. The statement~\ref{l36i} is proved.

Before proving~\ref{l36ii} and~\ref{l36iii} note that the definition of $\wp$ implies that sequence $e_1^{(n)} = \wp(\phi/2; \Gamma_n)$ converges to a nonzero number $e = 2\pi^2/3\phi^2$. Now, Lemma~\ref{l35}~\ref{l35ii} implies that $g_2(\Gamma_n) \rightarrow g_2(G)$ and $g_3(\Gamma_n) \rightarrow g_3(G)$. Since the polynomial $4x^3 - g_2(G) x - g_3(G)$ has multiple roots, $|e_2^{(n)} - e_3^{(n)}| \le \min\{|e_1^{(n)} - e_2^{(n)}|, |e_1^{(n)} - e_3^{(n)}|\}$, and roots are continuous with respect to the coefficients, we obtain that $|e_2^{(n)} - e_3^{(n)}| \rightarrow 0$. Thus, sequences $e_2^{(n)}$ and $e_3^{(n)}$ converge to $-e/2$.

Formula~\eqref{eq_delta} implies that
$$
    \Delta(\Gamma_{n+1}) = \frac{\Delta(\Gamma_n)^2}{4096 (e_1^{(n)} - e_2^{(n)})^3 (e_1^{(n)} - e_3^{(n)})^3}.
$$
Since the denominator in this equality has a nonzero limit we easily obtain that $\{\Delta(\Gamma_n)\}_{n \in \mathbb N}$ converges to $0$ quadratically fast.

Finally, statement~\ref{l36iii} follows from~\ref{l36ii} and the equality
$$
    (e_2^{(n)} - e_3^{(n)})^2 = \frac{\Delta(\Gamma_n)}{16(e_1^{(n)} - e_2^{(n)})^2 (e_1^{(n)} - e_3^{(n)})^2}.
$$

\end{proof}
\begin{theorem}\label{t32}
    Let $\Gamma_0 \supset \Gamma_1 \supset \Gamma_2 \supset \dots$ and $(e_1^{(n)}, e_2^{(n)}, e_3^{(n)}), n \in \NN$ satisfy the condition of Lemma~\ref{l33}. Then the sequences $\{g_2(\Gamma_n)\}_{n \in \NN}$ and $\{g_3(\Gamma_n)\}_{n \in \NN}$ converge quadratically fast.
\end{theorem}
\begin{proof}
    For convenience we fix $n$ and let $e_j = e^{(n)}_j$, $j = 1,2,3$. Formulae~\eqref{eq_g2},~\eqref{eq_g3}, and~\eqref{eq_InvFromRoots} imply that
    $$g_2(\Gamma_{n+1}) - g_2(\Gamma_n) = \frac{e_1^2}{4} + (e_1 - e_2)(e_1 - e_3) - 2(e_1^2+e_2^2 + e_3^2)= -\frac{5}{4}(e_2 - e_3)^2,$$
    $$g_3(\Gamma_{n+1}) - g_3(\Gamma_n) = -\frac{e_1^3}{8} + \frac{e_1}{2}(e_1 - e_2)(e_1 - e_3) - 4e_1 e_2 e_3 = \frac{7e_1}{8}(e_2 - e_3)^2.$$

    Now the statement follows from Lemma~\ref{l36}~\ref{l36iii}.
\end{proof}
    \section{Computation of periods and the Abel map}
    \label{sec:periods}
    As usual, we denote by $\CC^{(n)}$ the set of all unordered $n$-tuples of complex numbers. We define the mapping $\mathscr L : \CC \times \CC^{(2)} \rightarrow \CC \times \CC^{(2)}$ by the following condition:
    $(f_1,\{f_2,f_3\}) = \mathscr L(e_1,\{e_2,e_3\})$ if
    $$
        f_1 = -\frac{e_1}{2},\;\; f_1 + f_2 + f_3 = 0,\;\;16(f_2 - f_1)(f_3 - f_1) = (e_2 - e_3)^2.
    $$
    As was noted after Proposition~\ref{p31}, $\mathscr L$ is correctly defined. Also, Proposition~\ref{p31}~\ref{p31ii} implies that if $e_1,e_2,e_3$ are distinct roots of the polynomial $4x^3 - g_2(\Gamma)x - g_3(\Gamma)$ for a lattice $\Gamma$, then $\mathscr L(e_1, \{e_2, e_3\})$ consists of the roots corresponding to a subgroup in $\Gamma$ of index $2$. Different choices of a separate root correspond to different subgroups in $\Gamma$. We will also use the mapping $\mathscr F:\CC \times \CC^{(2)} \rightarrow \CC^{(3)}$ that forgets about the selected element, i.e. $\mathscr F(e_1, \{e_2, e_3\}) = \{e_1, e_2, e_3\}$.

    We denote by $\mathfrak G \subset \CC^{(3)}$ the set of all unordered triples of complex numbers $\{e_1, e_2, e_3\}$ such that $e_1 + e_2 + e_3 = 0$ and there is a shuffle $\{j, k, l\} = \{1, 2, 3\}$ such that $|e_j - e_k| > |e_k - e_l|$ and $|e_j - e_l| > |e_k - e_l|$. That is, Proposition~\ref{p33} implies that triples from $\mathfrak G$ correspond to lattices that have exactly one optimal subgroup of index $2$. We define $\mathscr S:\mathfrak G \rightarrow \CC \times \CC^{(2)}$ to be the selection of $e_j$, i.e. $\mathscr S(\{e_1, e_2,e_3\}) = (e_j, \{e_k, e_l\})$. On the set $\mathfrak G$ we define the mapping $\widetilde{\mathscr {L}}: \mathfrak G \rightarrow \CC \times \CC^{(2)}$ that is given as $\widetilde{\mathscr{L}}(e) = \mathscr L(\mathscr S(e))$. Corollary~\ref{cp33} implies that $\mathscr F \circ \widetilde{\mathscr{L}}$ maps $\mathfrak G$ into itself.

    \begin{remark}
        The operations $\mathscr F$ and $\mathscr S$ are necessary for the convenient expression of the computational methods below, since the choice of the optimal subgroup of index $2$ in a lattice on each step always requires a reordering of the roots. Indeed, without the reordering the iterations of the Landen transformation $\mathscr L$ will always be not an optimal choice. More precisely, it can be verified that $\mathscr L(\mathscr L(e))$ is equal to $e/4$ up to a reordering.
    \end{remark}

    \begin{lemma}\label{l41}
        Let $\Gamma$ be a lattice and $(e_1,e_2,e_3)$ be the properly ordered triple of roots of the polynomial $4x^3 - g_2(\Gamma)x - g_3(\Gamma)$. Also let $\omega_1,\omega_2$ be a reduced basis in $\Gamma$ that satisfies the conditions of Corollary~\ref{ctRoots}. Finally, let $(f_1,\{f_2, f_3\}) = \mathscr L(e_2, \{e_1, e_3\})$.
        \begin{enumerate}[label=(\roman*)]
            \item\label{l41i}
                The subgroup $\hat{\Gamma} = \Span\{\omega_2, 2\omega_1\}$ of index $2$ in $\Gamma$ has $\{f_1, f_2, f_3\}$ as the roots of the corresponding polynomial. Also $f_1 = \wp(\omega_1, \hat{\Gamma})$.
            \item\label{l41ii}
                Either $\omega_2$, or $2\omega_1$ is an element in $\hat{\Gamma} \setminus\{0\}$ with minimal absolute value.
                Moreover, if $2\omega_1$ has minimal absolute value, then the basis $2\omega_1, \omega_2$ is reduced. Finally, if $\omega_2$ has minimal absolute value, then one of the bases $(\omega_2, 2\omega_1)$, $(\omega_2, 2\omega_1 + \omega_2)$, or $(\omega_2, 2\omega_1 - \omega_2)$ is reduced.
            \item\label{l41iii}
                $2|\omega_1| \le |\omega_2|$ if and only if $|f_1 - f_2| \ge |f_2 - f_3|$ and $|f_1 - f_3| \ge |f_2 - f_3|$.
        \end{enumerate}
    \end{lemma}
    \begin{proof}
        The statement~\ref{l41i} is a direct consequence of Proposition~\ref{p32}~\ref{p32ii}. To prove~\ref{l41ii} assume that $2|\omega_1| > \inf\{|w|: w \in \hat{\Gamma} \setminus\{0\}\}$. Then, obviously, the minimal absolute value element in $\hat{\Gamma}$ belongs to $\Gamma \setminus \omega_1\ZZ$. It is easy to conclude, that $|\omega_2| = \inf\{|w|: w \in \hat{\Gamma} \setminus\{0\}\}$. Now consider the case $2|\omega_1| = \inf\{|w|: w \in \hat{\Gamma} \setminus\{0\}\}$. Since $\hat{\Gamma} \setminus 2\omega_1\ZZ \subset \Gamma \setminus \omega_1\ZZ$ we obtain that also $|\omega_2| = \inf\{|w|: w \in \hat{\Gamma} \setminus 2\omega_1\ZZ\}$, i.e. basis $2\omega_1, \omega_2$ is reduced. Finally, assume $|\omega_2| = \inf\{|w|: w \in \hat{\Gamma} \setminus\{0\}\}$. Without loss of generality we put $\omega_1 = 1$. Then $1\le |\omega_2| \le 2$ and $|\mathrm{Re}\:\omega_2| \le 1/2$. It is clear that there exists a reduced basis in $\hat{\Gamma}$ of the form $\omega_2, 2 + k\omega_2$ for some $k \in \ZZ$. Moreover, it is sufficient to satisfy $|\mathrm{Re}\:(2/\omega_2 + k)| \le 1/2$. The conditions on $\omega_2$ easily imply that $|\mathrm{Re}\:(2/\omega_2)| \le 1$, so it is possible to find $k$ from the set $\{-1, 0, 1\}$.

        Now we prove~\ref{l41iii}. Suppose that $2|\omega_1| \le |\omega_2|$. Then the basis $2\omega_1, \omega_2$ is reduced and the statement follows from Theorem~\ref{tRoots}~\ref{tRootsi}. To prove the converse assume that $|f_1 - f_2| \ge |f_2 - f_3|$ and $|f_1 - f_3| \ge |f_2 - f_3|$. If $2|\omega_1| > |\omega_2|$, then by~\ref{l34ii} $2\omega_1$ is either the second element, or the sum of elements in a reduced basis in $\Gamma$. In both cases Theorem~\ref{tRoots} combined with the assumption implies that $2\omega_1$ has the smallest absolute value among the elements of $\hat{\Gamma}\setminus\{0\}$, which contradicts $2|\omega_1| > |\omega_2|$.
    \end{proof}
    Now we can formulate an algorithm that computes a reduced basis in a lattice $\Gamma$ given $g_2(\Gamma)$ and $g_3(\Gamma)$. The idea can be formulated as follows: to compute the smallest period just iterate the transformation $\widetilde{\mathscr {L}}$ to choose an optimal subgroup of index 2 until the corresponding lattice is close enough to a rank-$1$ group. Then we can use the formula that relates remaining period to the roots. In order to find the second period in the basis we iterate the Landen transformation to find not an optimal subgroup of index $2$ (so the smallest period multiplies by $2$ and the second period remains the same) until the smallest period of that lattice appears to be the required second period of $\Gamma$. After that we can just find the smallest period in the obtained lattice.
    \begin{algo}\label{alg41}
    \begin{enumerate}
        \item Calculate a properly ordered triple $(e_1,e_2,e_3)$ of distinct roots of the polynomial $4x^3 - g_2(\Gamma)x - g_3(\Gamma)$.
        \item Calculate $f^{(0)} = \mathscr L(e_1,\{e_2,e_3\})$ and $h^{(0)} = \mathscr L(e_2,\{e_1,e_3\})$.
        \item Calculate $f^{(n)} = \widetilde{\mathscr {L}}(\mathscr F (f^{(n-1)}))$ until the difference between two closest roots in the triple $\mathscr F(f^{(n)})$ is sufficiently small. Let $N$ denote the number of iterations. So $f^{(N)} = (f^{(N)}_1, \{f^{(N)}_2, f^{(N)}_3\})$, where $f^{(N)}_1 \approx -e/2$ and $\{f^{(N)}_2, f^{(N)}_3\} \approx \{e, -e/2\}$. Now $\omega_1 = i\pi /\sqrt{3f^{(N)}_1}$ is an approximation for an element in $\Gamma\setminus\{0\}$ with the smallest absolute value.
        \item Assume that the conditions $|h^{(n-1)}_1 - h^{(n-1)}_2| \ge |h^{(n-1)}_2 - h^{(n-1)}_3|$ and $|h^{(n-1)}_1 - h^{(n-1)}_3| \ge |h^{(n-1)}_2 - h^{(n-1)}_3|$ hold. In this case we let $\tilde{h}^{(n-1)} \in \CC^3$ denote a properly ordered triple of the same numbers as in $h^{(n-1)}$ such that $h^{(n-1)}_1 = \tilde{h}^{(n-1)}_1$ and let $h^{(n)} = \mathscr L(\tilde{h}^{(n-1)}_2, \{\tilde{h}^{(n-1)}_1, \tilde{h}^{(n-1)}_3\})$. If the above conditions are not fulfilled we stop the iterations and let $k^{(0)} = \{h^{(n-1)}_1,h^{(n-1)}_2,h^{(n-1)}_3\}$.
        \item Use the same calculations as in the step $3$ to find an approximation for a nonzero element with the smallest absolute value in the lattice that corresponds to the roots $k^{(0)}$. This is an approximation for the second period of $\Gamma$.
    \end{enumerate}
    \end{algo}
    We also present an algorithm to compute $\mathcal A_{\Gamma}(x,y)$ given $g_2(\Gamma)$, $g_3(\Gamma)$, and a pair $(x,y) \in \CC^2$ such that $y^2 = 4x^3 - g_2(\Gamma)x - g_3(\Gamma)$. More precisely, the algorithm below computes some point $z \in \CC$ such that $\mathcal A_{\Gamma}(x,y) = z \;\mathrm{mod}\; \Gamma$.
    \begin{algo}\label{alg42}
    \begin{enumerate}
        \item Calculate a properly ordered $(e_1,e_2,e_3)$ triple of distinct roots of the polynomial $4x^3 - g_2(\Gamma)x - g_3(\Gamma)$.
        \item Let $e^{(1)} = \mathscr L(e_1, \{e_2, e_3\})$ and calculate $e^{(n)} = \widetilde{\mathscr{L}}(\mathscr F(e^{(n-1)}))$ until $|e^{(n)}_2 - e^{(n)}_3|$ is sufficiently small. Let $N$ denote the number of iterations and $\omega$ denote the approximation for the smallest period of $\Gamma$ (i.e. $\omega = i\pi / \sqrt{3e^{(N)}_1}$).
        \item Denote $x_0 = x$ and $y_0 = y$. Calculate a sequence $x_n$, $y_n$, $n = 1, \dots, N$ that satisfies
        $$
            x_{n-1} = x_n + \frac{(e^{(n)}_2 - e^{(n)}_1)(e^{(n)}_3 - e^{(n)}_1)}{x_n - e^{(n)}_1},
        $$
        $$
            y_{n-1} = y_n\left(1 - \frac{(e^{(n)}_2 - e^{(n)}_1)(e^{(n)}_3 - e^{(n)}_1)}{(x_n - e^{(n)}_1)^2}\right).
        $$
        On each iteration there are two possibilities for choosing $x_n$ (the value $y_n$ is determined by the choice of $x_n$). In order to make these sequences converge we require $x_n$ to be that solution of the first equation above, which is closer to $x_{n-1}$.
        \item As an approximation to $\mathcal A_{\Gamma}(x,y)$ we propose $z \;\mathrm{mod}\; \Gamma$, where
        $$
        z = -\frac{\omega}{\pi}\arctan\left(\frac{6\pi \omega^2 x_N + 2\pi^3}{3\omega^3 y_N}\right).
        $$
    \end{enumerate}
    \end{algo}

    For the analysis (in particular, the analysis of convergence) of the similar algorithms formulated in the setting of the complex AGM (which is an equivalent form of the Landen transformation) we refer to~\cite{Cremona}.
\section{Computation of Weierstrass functions}
    \label{sec:computation}

    We are ready to give an algorithm to compute values of the Weierstrass functions $\wp(z, \Gamma)$, $\wp'(z, \Gamma)$, $\zeta(z, \Gamma)$, $\sigma(z, \Gamma)$ given $z, g_2(\Gamma), g_3(\Gamma)$. We follow the ideas of the classical Landen method, that is, we compute a sequence of optimal subgroups of index $2$ until the corresponding Weierstrass functions are approximated well by the functions corresponding to a rank-$1$ additive subgroup. After that the approximations of the Weierstrass functions corresponding to $\Gamma$ are obtained by repeated application of formulae~\eqref{eq_PP2}-\eqref{eq_sigma_sigma}. The only difficulty in this approach is that we can only compute $\sigma^2(z, \Gamma)$ instead of $\sigma(z,\Gamma)$. This problem can be solved by a following trick: values $\wp(z, \Gamma)$, $\wp'(z, \Gamma)$, $\zeta(z, \Gamma)$, $\sigma(z, \Gamma)$ can be recovered from values $\wp(z/2, \Gamma)$, $\wp'(z/2, \Gamma)$, $\zeta(z/2, \Gamma)$, $\sigma(z/2, \Gamma)^2$ using duplication formulae (see \cite[Eq.~18.4.5-8]{AbramovitzStegun}).
    \begin{algo}\label{alg51}
    \begin{enumerate}
        \item Calculate a properly ordered $(e_1,e_2,e_3)$ triple of distinct roots of the polynomial $4x^3 - g_2(\Gamma)x - g_3(\Gamma)$.
        \item Put $e^{(1)} = \mathscr L(e_1, \{e_2, e_3\})$ and calculate $e^{(n)} = \widetilde{\mathscr {L}}(\mathscr F(e^{(n-1)}))$ until $|e^{(n)}_2 - e^{(n)}_3|$ is sufficiently small. Let $N$ denote the number of iterations and $\omega$ denote the approximation for the smallest period of $\Gamma$ (i.e. $\omega = i\pi /\sqrt{3e^{(N)}_1}$).
        \item Initialize $\tilde{\wp}_N, \tilde{\wp}'_N, \tilde{\zeta}_N, \tilde{\sigma}_N$ as
        $$
            \tilde{\wp}_N = \frac{\pi^2}{\omega^2}\left(\frac{1}{\sin\left(\frac{\pi z}{2\omega}\right)^2} - \frac{1}{3}\right),\;\;
            \tilde{\wp}'_N = -\frac{2\pi^3 \cos\left(\frac{\pi z}{2\omega}\right)}{\omega^3 \sin\left(\frac{\pi z}{2\omega}\right)^3}
        $$
        $$
            \tilde{\zeta}_N = \frac{\pi^2 z}{6 \omega^2} + \frac{\pi}{\omega} \cot\left(\pi \frac{z}{2\omega}\right),\;\;
            \tilde{\sigma}_N = \frac{\omega}{\pi}\exp\left(\frac{\pi^2 z^2}{24 \omega^2}\right)\sin\left(\pi \frac{z}{2\omega}\right).
        $$
        \item Compute $\tilde{\wp}_n, \tilde{\wp}'_n, \tilde{\zeta}_n, \tilde{\sigma}_n$ for $n = N-1, N-2, \dots, 0$ by the rules
        $$
            \tilde{\wp}_{n-1} = \tilde{\wp}_n + \frac{(e^{(n)}_2- e^{(n)}_1)(e^{(n)}_3 - e^{(n)}_1)}{\tilde{\wp}_n - e^{(n)}_1},
        $$
        $$
            \tilde{\wp}'_{n-1} = \tilde{\wp}'_{n}\left( 1 - \frac{(e^{(n)}_2- e^{(n)}_1)(e^{(n)}_3 - e^{(n)}_1)}{(\tilde{\wp}_n - e^{(n)}_1)^2}\right),
        $$
        $$
            \tilde{\zeta}_{n-1} = 2\tilde{\zeta}_{n} + \frac{1}{2} \frac{ \tilde{\wp}'_{n}}{\tilde{\wp}_n - e^{(n)}_1} + e^{(n)}_1 \frac{z}{2},
        $$
        $$
            \tilde{\sigma}_{n-1} = \exp\left(e^{(n)}_1 \frac{z^2}{4}\right) (\tilde{\wp}_n - e^{(n)}_1) \tilde{\sigma}_{n}^2.
        $$
        \item Finally, we propose the following approximations for the Weierstrass functions:
        $$
            \wp(z, \Gamma) \approx -2 \tilde{\wp}_0 + \left(\frac{6\tilde{\wp}^2_0 - g_2(\Gamma)/2}{2\tilde{\wp}'_0}\right)^2, $$ $$
            \wp'(z, \Gamma) \approx -\tilde{\wp}'_0 + \frac{6\tilde{\wp}^2_0 - g_2(\Gamma)/2}{4\tilde{\wp}'_0} \left(12\tilde{\wp}_0 - \left(\frac{6\tilde{\wp}^2_0 - g_2(\Gamma)/2}{\tilde{\wp}'_0}\right)^2 \right), $$ $$
            \zeta(z, \Gamma) = \approx 2\tilde{\zeta}_0 + \frac{6\tilde{\wp}^2_0 - g_2(\Gamma)/2}{2\tilde{\wp}'_0}, $$ $$
            \sigma(z, \Gamma) \approx -\tilde{\wp}'_0 \tilde{\sigma}^2_0
        $$
    \end{enumerate}
    \end{algo}

    The proof of convergence of the foregoing approximations is rather technical and we will only consider the approximation of $\wp$-function. The main tool is the following elementary lemma.
    \begin{lemma}\label{l51}
        Let $x \in \CC$ and let $\{f_n\}_{n \in \NN}$ be a sequence of meromorphic functions on $\CC$. Assume that there exists a neighborhood $U$ of $x$ such that functions $f_n$ do not have poles in $U$ for large $n$ and uniformly converge on $U$ to a holomorphic function as $n \rightarrow \infty$. Also assume that the sequences $\{x_n\}_{n \in \NN}$, $\{y_n\}_{n \in \NN}$ converge to $x$ quadratically fast. Then the sequence $\{f_n(x_n) - f_n(y_n)\}$ consists of finite complex numbers for large $n$ and converges to $0$ quadratically fast.
    \end{lemma}
    \begin{proof}
        Let $V$ denote a compact convex neighborhood of $x$ such that $\overline{V} \subset U$. Let $N \in \NN$ denote a number such that  $x_n,y_n \in U$ and $f_n$ is holomorphic on $U$ for $n \ge N$. It is clear that the values $f_n(x_n)$, $f_n(y_n)$, and $f_n(x)$ belongs to $\CC$ for $n \ge N$. Moreover, the functions $f'_n$ for $n \ge N$ are uniformly bounded on $V$ by some constant $C$. Thus, $|f_n(x_n) - f_n(x)| \le C|x_n - x|$ for $n \ge N$ and the sequence $\{f_n(x_n) - f_n(x)\}$ converges to $0$ quadratically fast. The statement follows as $f_n(x_n) - f_n(y_n) = (f_n(x_n) - f_n(x)) - (f_n(y_n) - f_n(x))$.
    \end{proof}

    Now we introduce the following notation. Let $e^{(n)}$, $N$, and $\omega$ be the same as in the step $(2)$ of the algorithm. To be more precise we use the notation $\omega_N$ instead of $\omega$, since it, obviously, depends on $N$. We denote by $\Gamma_n$ the lattice that corresponds to the roots $e^{(n)}$ and by $\phi_n(x)$ we denote the function
    $$
        \phi_n(x) = x + \frac{(e^{(n)}_2- e^{(n)}_1)(e^{(n)}_3 - e^{(n)}_1)}{x - e^{(n)}_1}.
    $$
    Finally, let $\Phi_n(x) = (\phi_1 \circ \dots \circ \phi_n)(x)$.

    \begin{proposition}\label{p51}
        Assume that $z \notin \Gamma$ and let $$x_N = \frac{\pi^2}{\omega_N^2}\left(\frac{1}{\sin\left(\frac{\pi z}{\omega_N}\right)^2} - \frac{1}{3}\right).$$ Then $\Phi_N(x_N)$ converges to $\wp(z, \Gamma)$ quadratically fast.
    \end{proposition}
    \begin{proof}
        It is clear that $\Gamma_0 = \Gamma$ and $\Gamma_n$ is optimal and has index $2$ in $\Gamma_{n-1}$. Thus, by Lemma~\ref{l36}~\ref{l36i}, $G = \bigcap_n \Gamma_n$ is spanned by a complex number $\Omega$. Moreover, an appropriate choice of signs of $\omega_N$ guarantees that $\omega_N$ converges to $\Omega$ quadratically fast (note that the $x_N$ does not change, when $\omega_N$ is replaced with $-\omega_N$). From this we can conclude that $x_N$ converges quadratically fast to $x = \wp(z, G)$. It is clear that the sequence $\wp(z, \Gamma_N)$ also converges to $x$ quadratically fast. The relation $\Phi_N(\wp(w, \Gamma_N)) = \wp(w, \Gamma)$ easily implies that the sequence of meromorphic functions $\Phi_N$ converges on the set $\CC\setminus \wp(\Gamma, G)$ to a holomorphic function. In particular, it converges in some neighborhood of $x$. Since
        $$
            \Phi_N(x_N) - \wp(z,\Gamma) = \Phi_N(x_N) - \Phi_N(\wp(z, \Gamma_N)),
        $$
        Lemma~\ref{l51} is proved.
    \end{proof}
    \begin{remarks}
        \begin{enumerate}
            \item All the algorithms that we have presented require to compute some number of optimal subgroups (more precisely, roots of the corresponding polynomials) of undex $2$. If it is required to perform significant amount of computations with a fixed elliptic curve, it is reasonable to precompute those roots and store them. Since usual number of iterations that is required to achieve good accuracy is quite small, the memory cost is negligible.
            \item Algorithm~\ref{alg51} can be optimized if it is not required to approximate all four functions at once. At first we note that the duplication step is necessary only for computation of $\sigma$. In addition, the computation of $\wp$ does not depend on other functions, $\wp'$ depends only on $\wp$, and both $\zeta$ and $\sigma$ depend on $\wp$ and $\wp'$ (but are independent from each other). Thus, to compute only a subset of values $\{ \wp(z, \Gamma), \wp'(z, \Gamma), \zeta(z, \Gamma), \sigma(z, \Gamma)\}$ one can perform a simpler version of the algorithm.
            \item In practice of computations with special functions it is often required to calculate the derivatives of these functions with respect to their variables and parameters. Fortunately, the Weierstrass functions satisfy specific differential equations~\cite[Eqs.~18.6.1-24]{AbramovitzStegun} that allow to compute all their derivatives with respect to $z$, $g_2$, and $g_3$ given only their values.
        \end{enumerate}
    \end{remarks}
\section{Numerical experiments}\label{sec:experiments}
    \subsection{Demonstration of quadratic convergence}
    Table~\ref{InvConv} shows the convergence of values $g_2$, $g_3$, and $\Delta$ of the lattices $\Gamma_0 \supset \Gamma_1 \supset \dots$ obtained by choosing an optimal subgroup of index $2$ on each step starting from an initial lattice $\Gamma = \Gamma_0$. More precisely, the Table~\ref{InvConv} numerically verifies statements of Lemma~\ref{l36}~\ref{l36ii} and Theorem~\ref{t32}. In the experiment we consider $g_2(\Gamma) = 3 + i$, $g_3(\Gamma) = 2$. For sequences that converge to a non-zero limit we show $30$ decimal places and underline those that do not change under further iterations.
    \begin{center}
        \begin{table}
        \begin{tabular}{|c||c|}
            \hline $n$ & $g_2(\Gamma_n)$ \\ \hline
            $0$ &  $3 + i$ \\
            $1$ & $\underline{3.75}4046867215436982426029182236 + \underline{0.54}0233967914303556235718229303i$ \\
            $2$ & $\underline{3.753771977}059587664114076064651 + \underline{0.54105649}4694848332981391043677i $ \\
            $3$ & $\underline{3.753771977783970498856}515753866 + \underline{0.54105649509239637214}2231763369i$ \\
            $4$ & $\underline{3.753771977783970498856026746202} + \underline{0.541056495092396372141662941563}i$ \\
            \hline $n$ & $g_3(\Gamma_n)$ \\ \hline
            $0$ &  $2$ \\
            $1$ & $\underline{1.388}499235514097862630349344347 + \underline{0.30}3503045561126645130957672495i$ \\
            $2$ & $\underline{1.388761317}907632838227691307107 + \underline{0.302872794}924673800604147812848i$ \\
            $3$ & $\underline{1.388761317361341232445}441066859 + \underline{0.302872794571811322640}063572398i$ \\
            $4$ & $\underline{1.388761317361341232445792939849} + \underline{0.302872794571811322640537643014}i$ \\
            \hline $n$ & $\Delta(\Gamma_n)$ \\ \hline
            $0$ &  $-90 + 26i$ \\
            $1$ & $0.0513671601 - 0.0736732833i$ \\
            $2$ & $10^{-8}(-6.0337705864 - 6.0680444150i)$ \\
            $3$ & $10^{-20}(3.1944965545 + 7.0811930101i)$ \\
            $4$ & $10^{-44}(-1.8537859902 + 6.1278114526i)$ \\ \hline

        \end{tabular}
        \caption{Convergence of Weierstrass invatiants}\label{InvConv}
        \end{table}
    \end{center}

    The convergence of Algorithms~\ref{alg41} and~\ref{alg42} is demonstrated in Table~\ref{PerConv}. As before, we denote by $N$ the number of iterations of the Landen transformation in these algorithms; let $\omega_N$ denote the obtained approximation of the first period. We continue to use the lattice $\Gamma$ defined above and consider the point $(x,y) = (1, i 2^{1/4} \exp(i\pi/8))$ on the curve $y^2 = 4x^3 - (3 + i)x - 2$. By $z_N$ we denote the approximation of the $\mathcal A_{\Gamma}(x,y)$.

    \begin{center}
        \begin{table}
        \begin{tabular}{|c||c|}
            \hline $N$ & $\omega_N$ \\ \hline
            $1$ & $\underline{2.4}38686216965391972931889039948 - \underline{0}.105955591501509972308694592494i$ \\
            $2$ & $\underline{2.41753}3084489739068968559720359 - \underline{0.0865}27699052746187490062524284i$ \\
            $3$ & $\underline{2.417537043}106790993092839472406 - \underline{0.08655507279}1232588159988669113i$ \\
            $4$ & $\underline{2.4175370430818008602841}29467153 - \underline{0.0865550727995970630460}98367581i$ \\
            $5$ & $\underline{2.417537043081800860284148042662} - \underline{0.086555072799597063046083291895}i$ \\
            \hline $N$ & $z_N$ \\ \hline
            $1$ & $\underline{1.1}48555533478147362319765496898 + \underline{0.16}5542168411567103609102081635i$ \\
            $2$ & $\underline{1.1355}03055177661590826945142841 + \underline{0.1682}41922881856989982768683181i$ \\
            $3$ & $\underline{1.1355110948}76954045535981843541 + \underline{0.1682319645}15852775813630885200i$ \\
            $4$ & $\underline{1.13551109486898465067558}5138964 + \underline{0.1682319645066226442821}84848575i$ \\
            $5$ & $\underline{1.135511094868984650675588970809} + \underline{0.168231964506622644282195234558}i$ \\ \hline
        \end{tabular}
        \caption{Convergence of numerical approximations of period and the Abel map}\label{PerConv}
        \end{table}
    \end{center}

    Finally we show a similar table (namely, Table~\ref{WeierConv}) for numerical approximations of the Weierstrass functions using Algorithm~\ref{alg51}. Moreover, to verify that the functions $\wp$ and $\wp'$ constitute the inverse of the Abel map we calculate the Weierstrass functions at $z = z_5 \approx \mathcal A_{\Gamma}(x,y)$ (the value of $z_5$ is given in Table~\ref{PerConv}).

    \begin{center}
        \begin{table}
        \begin{tabular}{|c||c|}
            \hline $N$ & $\wp(z, \Gamma)$ \\ \hline
            $1$ & $\underline{0.9}58026049179506041264531653969 + \underline{0.00}4519165676065043298534371539i$ \\
            $2$ & $\underline{1.0000}28837131162405084132407776 - \underline{0.0000}30226500252411472240586290i$ \\
            $3$ & $\underline{0.9999999999}76565657334803373714 - \underline{0.0000000000}32173573518069468236i$ \\
            $4$ & $\underline{1.00000000000000000000000}9663711 + \underline{0.0000000000000000000000}34817638i$ \\
            $5$ & $\underline{0.999999999999999999999999999999} + \underline{0.000000000000000000000000000000}i$ \\
            \hline $N$ & $\wp'(z, \Gamma)$ \\ \hline
            $1$ & $-\underline{0}.608720652206082355023823356004 + \underline{1}.124401544733847722373159855446i$ \\
            $2$ & $-\underline{0.45}4989563842651342413589094212 + \underline{1.098}565867378753359267771366508i$ \\
            $3$ & $-\underline{0.455089860}656029273164117412507 + \underline{1.098684113}353679030979521996790i$ \\
            $4$ & $-\underline{0.4550898605622273413043}14405895 + \underline{1.098684113467809966039}928107431i$ \\
            $5$ & $-\underline{0.455089860562227341304357757822} + \underline{1.098684113467809966039801195240}i$ \\
            \hline $N$ & $\zeta(z, \Gamma)$ \\ \hline
            $1$ & $\underline{0.78}8943935813327461013174357466 - \underline{0.20}5752011501413428445974864891i$ \\
            $2$ & $\underline{0.78355}7307095215698756718513021 - \underline{0.206}404251262761178141034497064i$ \\
            $3$ & $\underline{0.783555262}397155738789668978180 - \underline{0.206399816}303039582701103284365i$ \\
            $4$ & $\underline{0.7835552624125877530424}74199703 - \underline{0.2063998162856248000766}13912688i$ \\
            $5$ & $\underline{0.783555262412587753042456275712} - \underline{0.206399816285624800076666108370}i$ \\
            \hline $N$ & $\sigma(z, \Gamma)$ \\ \hline
            $1$ & $\underline{1.119}535114208134786254589186127 + \underline{0.1}40020472376646390837218752895i$ \\
            $2$ & $\underline{1.11947}6734388964409646347819029 + \underline{0.13978}6796169078699277398913537i$ \\
            $3$ & $\underline{1.1194741359}28442587178028125697 + \underline{0.1397886896}82872474343814804895i$ \\
            $4$ & $\underline{1.11947413593212617223716}8580364 + \underline{0.1397886896914695257773}52475490i$ \\
            $5$ & $\underline{1.119474135932126172237167916856} + \underline{0.139788689691469525777332568971}i$ \\ \hline
        \end{tabular}
        \caption{Convergence of numerical approximations of Weierstrass functions}\label{WeierConv}
        \end{table}
    \end{center}
    \FloatBarrier
    \subsection{An application to a conformal mapping problem}
    In \cite{Smirnov} the Weierstrass functions were applied to solve and analyze the conformal mapping problem for the region given in Fig.~\ref{fig1}.
    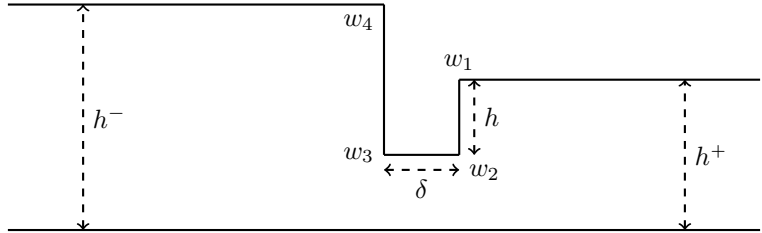
\begin{figure}[h!]
	\centering
\begin{tikzpicture}
	\draw[black,thick] (-5,-1) -- (5,-1);
	\draw[black,thick] (-5,2) -- (0,2) node [pos=1,below left] {$w_4$};
	\draw[black,thick] (0,2) -- (0,0) node [pos=1,left] {$w_3$};
	\draw[black,thick] (0,0) -- (1,0) node [pos=1,below right] {$w_2$};
	\draw[black,thick] (1,0) -- (1,1) node [pos=1,above] {$w_1$};
	\draw[black,thick] (1,1) -- (5,1);
	\draw[<->,black,dashed,thick] (0,-0.2) -- (1,-0.2) node [pos=0.5,below] {$\delta$};
	\draw[<->,black,dashed,thick] (1.2,0) -- (1.2,1)  node [pos=0.5,right] {$h$};
	\draw[<->,black,dashed,thick] (-4,-1) -- (-4,2) node [pos=0.5,right] {$h^-$};
	\draw[<->,black,dashed,thick] (4,-1) -- (4,1) node [pos=0.5,right] {$h^+$};
\end{tikzpicture}
	\caption{The domain $\Omega$.}\label{fig1}
\end{figure}
    It was proved that there exists a rectangular lattice $\Gamma$ (that is, $\Gamma = \Span\{\omega_1, \omega_2\}$, where $\omega_1, \omega_2/i > 0$) and parameters $D, z^+, z^-$ such that the following statements hold.
    \begin{enumerate}
        \item $D$ is purely imaginary and $0 < z^-/i < z^+/i < \omega_2/i$.
        \item The function 
$$
Q(z) = Dz + \frac{h^-}{\pi}\ln\left(\frac{\sigma(z-z^-)}{\sigma(z+z^-)}\right) - \frac{h^+}{\pi}\ln\left(\frac{\sigma(z-z^+)}{\sigma(z+z^+)}\right) - i(h^- - h^+),
$$
        conformally maps the rectangle with vertices $0, \omega_2/2, (\omega_2 - \omega_1)/2, -\omega_1/2$ onto $\Omega$ (here $\sigma(z)$ is short for $\sigma(z, \Gamma)$).
    \end{enumerate}
    A necessary and sufficient condition for the parameters $\omega_1, \omega_2, D, z^+, z^-$ to satisfy the statement 2 above was obtained in \cite[Eq.~(3.8)]{Smirnov}. This condition is given in the form of 4 real (nonlinear) equations. To reduce the parameters (i.e., to make the number of equations match the number of unknowns) it was proposed to consider a one-parameter family of elliptic curves instead of a two-parameter family of lattices. The curve corresponding to a parameter $\gamma \in (-1/6, 1/6)$ is given by the equation $y^2 = 4(x - e_1(\gamma))(x - e_2(\gamma))(x - e_3(\gamma))$, where $e_1(\gamma) = \gamma - 1/2$, $e_2(\gamma) = -2\gamma$, and $e_3(\gamma) = \gamma + 1/2$.

    In~\cite[Sec.~4.2]{Smirnov} it was shown that parameters $\gamma, D, z^+, z^-$ have limiting values as $\delta \rightarrow 0$ and other parameters of the domain $\Omega$ being fixed. Moreover, the conformal mapping also survives under the passage to the limit. It was shown that $D \rightarrow ih/3$, $\gamma \rightarrow -1/6$ (that is, the limit curve is singular or, equivalently, the limit discrete subgroup is no longer a lattice) and $1/6 + \gamma \sim C\sqrt{\delta}$ with an appropriate constant $C$. However, the computational approach to the Weierstrass functions in~\cite{Smirnov} does not allow to solve numerically the system of equations on parameters for small enough $\delta$. Using the Landen-type method we finally can numerically confirm the theoretical estimations on behaviour $\gamma, D, z^+, z^-$ as $\delta \rightarrow 0$.

    In Fig.~\ref{fig2} and~\ref{fig3} we show the parameters $\gamma, D, z^+, z^-$ of the conformal mapping as functions of $\delta$ with fixed $h = 0.6$, $h^+ = \pi$, and $h^- = \pi + 0.5$. In Fig.~\ref{fig2} it is clear that the parameters converge as $\delta\rightarrow 0$, and in Fig.~\ref{fig3} we compare the asymptotics of $\gamma + 1/6$ and $\sqrt{\delta}$ as $\delta \rightarrow 0$.

    \begin{figure}[ht!]
    	\centering
    	\includegraphics[width=0.7\textwidth]{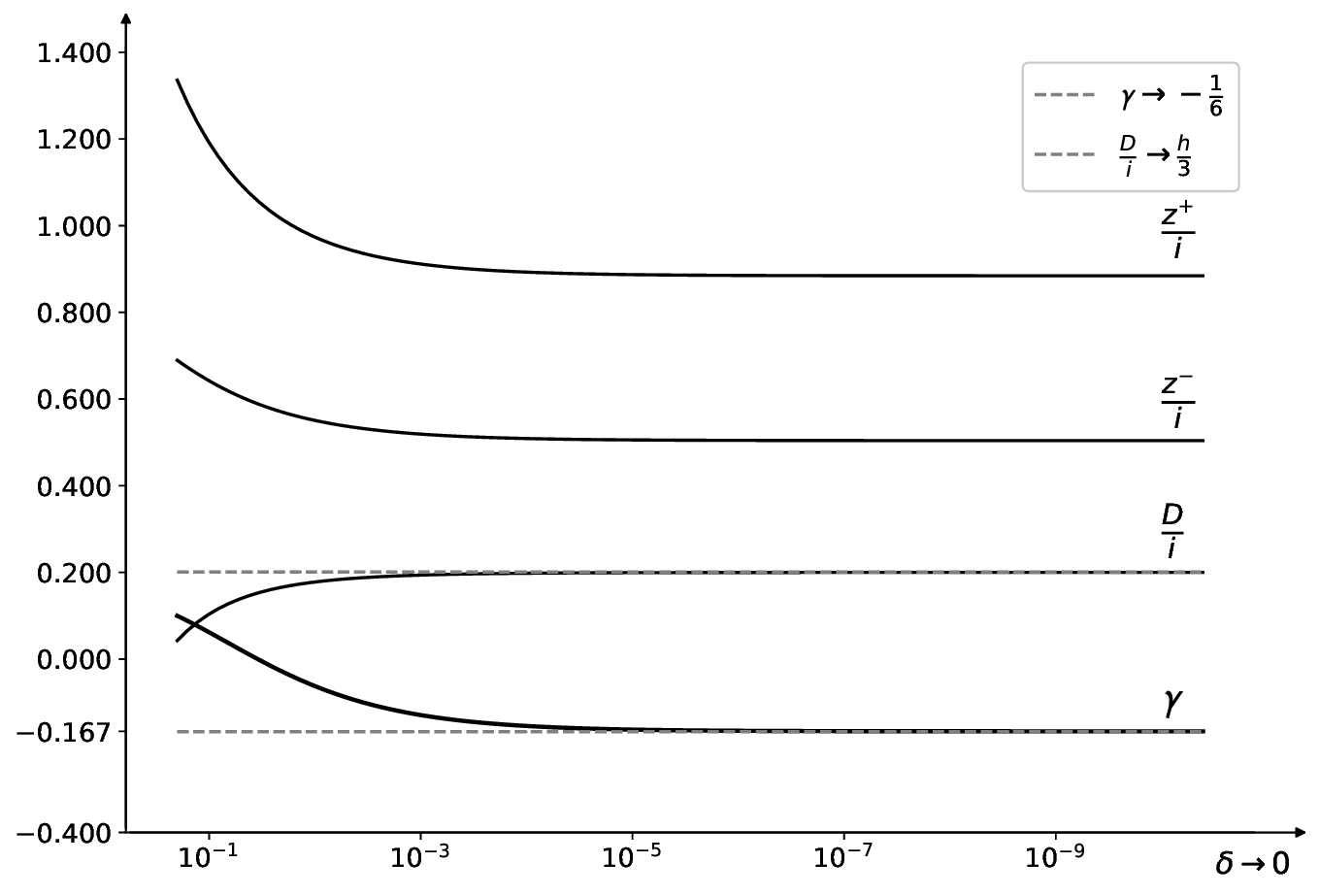}
    	\caption{The behaviour of $\gamma, D, z^+, z^-$ as functions of $\delta$.}
    	\label{fig2}
    \end{figure}

    \begin{figure}[ht!]
		\centering
		\includegraphics[width=0.7\textwidth]{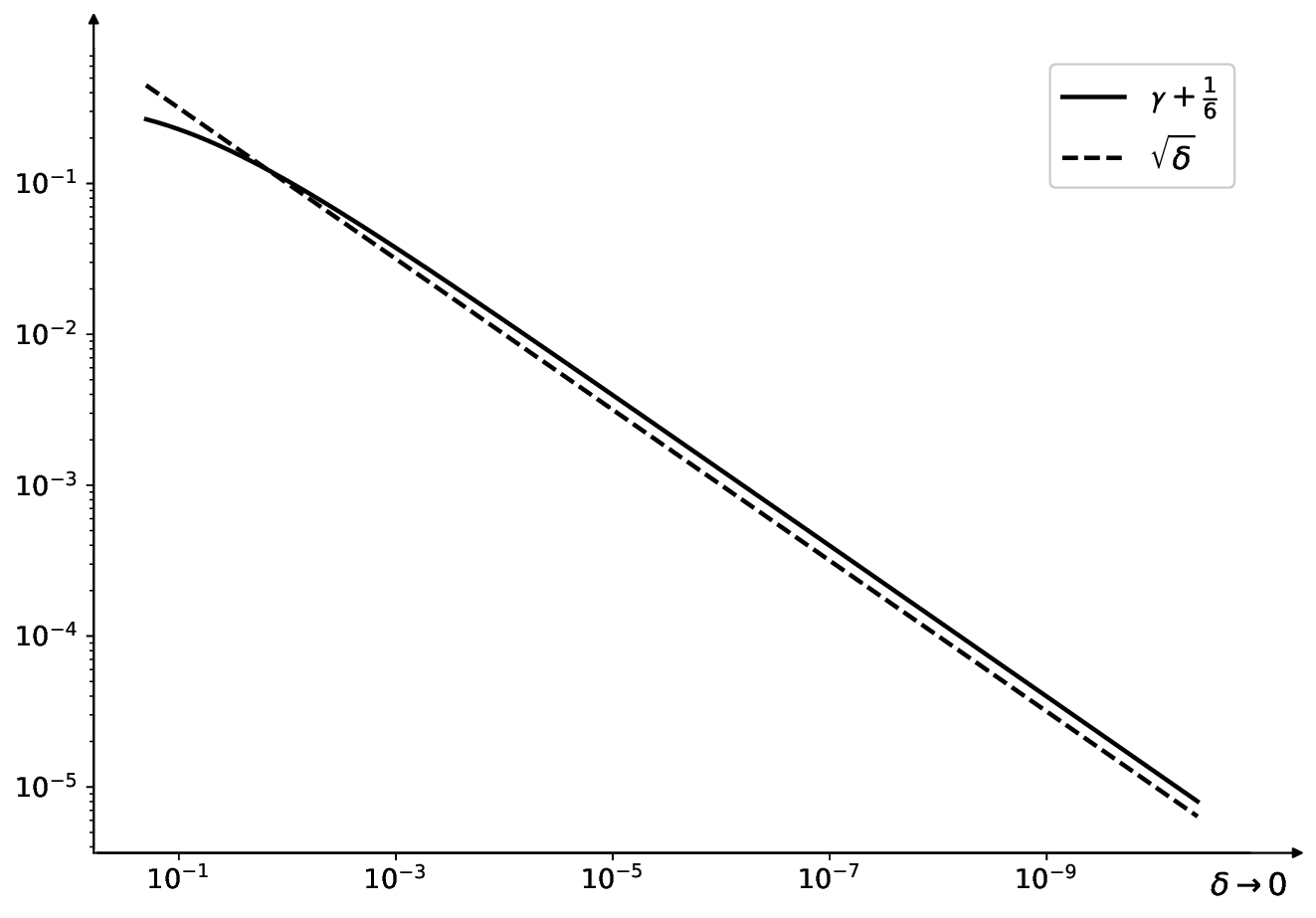}
		\caption{Asymptotics of $\gamma + 1/6$ as $\delta \rightarrow 0$.}
		\label{fig3}
	\end{figure}
	\FloatBarrier
	\section*{Acknowledgments}
		The authors express gratitude towards A.B.~Bogatyrev and S.A.~Goreinov for valuable discussions. M.S.~Smirnov  was supported by the Moscow Center of Fundamental and Applied Mathematics at INM RAS (Agreement with the Ministry of Education and Science of the Russian Federation No.075-15-2022-286).

    \printbibliography
\end{document}